\documentclass[12pt]{amsart}

\usepackage{amsrefs}
\usepackage{amssymb}
\usepackage{fancyhdr}
\usepackage{bbm}
\usepackage{xcolor}
\usepackage{hyperref}
\hypersetup{
  colorlinks=true,
  linkcolor=blue,
  citecolor=blue
}

\newtheorem{theorem}{Theorem}
\newtheorem*{theorem*}{Theorem}
\newtheorem{lemma}[theorem]{Lemma}
\newtheorem{proposition}[theorem]{Proposition}
\newtheorem{claim}[theorem]{Claim}

\newtheorem{question}[theorem]{Question}

\theoremstyle{definition}
\newtheorem{definition}[theorem]{Definition}
\newtheorem*{definition*}{Definition}
\newtheorem*{lemma*}{Lemma}
\usepackage{graphicx}
\usepackage{pstricks, enumerate, pst-node, pst-text, pst-plot}

\numberwithin{equation}{section}
\numberwithin{theorem}{section}

\newcommand{\Z}{\mathbb{Z}}

\newcommand{\eps}{\varepsilon}

\usepackage{xparse}
\DeclareDocumentCommand\Pr{ m g }{\ensuremath{
    {   \IfNoValueTF {#2}
      {\mathbb{P}\left[{#1}\right]}
      {\mathbb{P}\left[{#1}\middle\vert{#2}\right]}%
    }
}}
\DeclareDocumentCommand\E{ m g }{\ensuremath{
    {   \IfNoValueTF {#2}
      {\mathbb{E}\left[{#1}\right]}
      {\mathbb{E}\left[{#1}\middle\vert{#2}\right]}%
    }
}}

\def\aut{\mathrm{Aut}}

\begin{document}

\title[]{Characteristic measures of symbolic dynamical systems}

\author[]{Joshua Frisch and Omer Tamuz}
\address{California Institute of Technology}


\thanks{J.\ Frisch was funded by an NSF Grant (DMS-1464475). O.\ Tamuz was supported by a grant from the Simons Foundation (\#419427), a Sloan research fellowship, a BSF award (\#2018397), and an NSF CAREER award (DMS-1944153).}
\date{\today}

\begin{abstract}
  A probability measure is a characteristic measure of a topological dynamical system if it is invariant to the automorphism group of the system. We show that zero entropy shifts always admit characteristic measures. We use similar techniques to show that automorphism groups of minimal zero entropy shifts are sofic.
\end{abstract}

\maketitle

\section{Introduction}

Let $(G,X)$ be a topological dynamical system:  a jointly continuous action of a topological group $G$ on a compact Hausdorff space $X$. A homeomorphism $\varphi$ of $X$ is an automorphism of $(G,X)$  if $g \circ \varphi = \varphi \circ g$ for all $g \in G$. We denote by $\aut(G,X)$ the group of automorphisms, equipped with the compact-open topology. A Borel probability measure $\nu$ on $X$ is {\em invariant} if $g_*\nu = \nu$ for all $g \in G$.
\begin{definition}
  A Borel probability measure $\nu$ on $X$ is {\em characteristic} if $\varphi_*\nu = \nu$ for all $\varphi \in \aut(G,X)$.
\end{definition}
Note that characteristic measures are not necessarily invariant, and invariant measures are not necessarily characteristic. However, when $G$ is abelian then $G$ is a subgroup of $\aut(G,X)$, and hence every characteristic measure is $G$-invariant; this is not true for general $G$. When $G$ is amenable then $(G,X)$ admits invariant measures, and moreover, if there are characteristic measures, then there are characteristic invariant measures. Likewise, if $\aut(G,X)$ is amenable then there are characteristic measures, and if there are invariant measures then there are characteristic invariant measures. This follows from the fact that $G$ (resp., $\aut(G,X)$) acts affinely on the compact, convex set of characteristic (resp., invariant) measures.


In this paper we will focus on {\em symbolic dynamical systems}, or shifts, and restrict our attention to finitely generated $G$. Let $A$ be a finite alphabet. The {\em full shift} is the dynamical system $(G,A^G)$, where $A^G$ is equipped with the product topology and the action is by left translations. A {\em shift} $(G,\Sigma)$ is a subsystem of $(G,A^G)$, with $\Sigma$ a closed, $G$-invariant subset of $A^G$.

The automorphism groups of shifts are always countable~\cite{hedlund1969endomorphisms}. Even in the simplest case that $G=\Z$, these groups exhibit rich structure; for example $\aut(\Z,2^\Z)$ contains the free group on two generators, as well as every finite group (see, e.g.,~\cite{boyle1988automorphism}).

Some shifts $(\Z,\Sigma)$ obviously admit characteristic measures: these include uniquely ergodic shifts, shifts with a unique measure of maximal entropy, shifts with periodic points (which include all shifts of finite type), and shifts with amenable automorphism groups. But since $\aut(\Z,\Sigma)$ is in general non-amenable, it is not obvious that every $(\Z,\Sigma)$ admits a characteristic measure. Indeed, we do not know if this holds.

Our main result concerns zero entropy shifts. To define the entropy of a shift, let $N_\Sigma(F)$, the {\em growth function} of $\Sigma$, assign to each finite $F \subset \Z$ the cardinality of the restriction of $\Sigma$ to $F$. The entropy of $\Sigma$ is given by
\begin{align*}
    h(\Sigma) = \inf_r \frac{1}{r}\log N_\Sigma(\{1,2,\ldots,r\}).
\end{align*}
\begin{theorem}
\label{thm:zero-ent}
  Let $(\Z,\Sigma)$ be a shift with $h(\Sigma)=0$. Then $(\Z,\Sigma)$ admits a characteristic measure.
\end{theorem}

Our proof techniques critically uses the zero entropy assumption, and thus leaves open the broader question:
\begin{question}
  Does every shift $(\Z,\Sigma)$ admit a characteristic measure?
\end{question}
We more generally do not know of any countable group $G$ and a shift $(G,\Sigma)$ that does not admit characteristic measures.

Recent work~\cite{cyr2015automorphism1,cyr2015automorphism,coven2015automorphisms,cyr2014automorphism, donoso2015automorphism, salo2014toeplitz, salo2014block} shows that ``small shifts'' have ``small automorphism groups.'' For example, minimal shifts with slow stretched exponential growth (that is, shifts with $N_\Sigma(F) =  O(e^{|F|^\beta})$ for  $\beta < 1/2$) have amenable automorphism groups, as shown by Cyr and Kra~\cite{cyr2015automorphism}. They conjecture that every minimal zero entropy shift has an amenable automorphism group. A proof of this conjecture would imply Theorem~\ref{thm:zero-ent} for minimal shifts.

Theorem~\ref{thm:zero-ent} is a consequence of the following, more general result that applies to finitely generated groups, and relates the existence of characteristic measures to the growth of the shift. Given a finitely generated group $G$, we fix a generating set, and denote by $B_r \subset G$ the ball of radius $r$, according to the corresponding word length metric. 
\begin{theorem}
\label{thm:growth}
Let $G$ be a finitely generated group.
Then every shift $(G,\Sigma)$ for which
\begin{align*}
    \liminf_r \frac{1}{r}\log N_\Sigma(B_r) = 0
\end{align*}
admits a characteristic measure.
\end{theorem}
Theorem~\ref{thm:zero-ent} is an immediate specialization of this result to the case $G=\Z$.

\subsection{Beyond symbolic systems}
It is simple to construct a dynamical system $(\Z,C)$, which is not symbolic, and which has no characteristic measures: simply let $\Z$ act trivially on the Cantor set $C$. This system admits no characteristic measures, since the Cantor set has no measure that is invariant to all of its homeomorphisms. 

Recall that a dynamical system $(G,X)$ is said to be {\em topologically transitive} if for every two non-empty open sets $U,W \subset X$ there is some $g \in G$ such that $gU \cap W \neq \emptyset$. The system $(G,X)$ is {\em minimal}
if $X$ has no closed, $G$-invariant sets. It is {\em free} if $g x \neq x$ for every $x\in X$ and every non-trivial $g \in G$; in the important case of $G = \Z$ every non-trivial minimal system is free.
\begin{question}
  Does there exist a non-trivial minimal topological dynamical system that does not admit a characteristic measure?
\end{question}
An example of a topologically transitive $\Z$-system without characteristic measures is the $\Z$ action by shifts on $C^\Z$, where $C$ is the Cantor set. 

Recall that $(G,X)$ is said to be proximal \cite{glasner1976proximal} if for every $x,y\in X$ there exists a net $(g_i)_i$ in $G$ such that $\lim_i g_i x=\lim_i g_i y$. Many constructions of dynamical systems without invariant measures are proximal (e.g., the Furstenberg boundary of non-amenable groups~\cite{furstenberg1963poisson, glasner1976proximal}). Hence the following claim highlights a tension that needs to be overcome in order to construct minimal systems without characteristic measures.
\begin{claim}
Let $(G,X)$ be a free system. Then $(\aut(G,X),X)$ is not proximal.
\end{claim}
\begin{proof}
Assume that $(\aut(G,X),X)$ is proximal. Then for each $x \in X$ and $g \in G$, there is a net $(\phi_i)_i$ such that $\lim_i \phi_i x = \lim_i \phi_i g x$. Since $G$ and $\aut(G,X)$ commute, and since the action is continuous, we have that $g \lim_i \phi_i x=\lim_i \phi_i x$. Hence $(G,X)$ is not free.
\end{proof}
\subsection{Soficity of automorphism groups}
We show the following result, using techniques that are similar to those used to prove Theorem~\ref{thm:zero-ent}.
\begin{theorem}
\label{cor:abelian}
Let $(\Z,\Sigma)$ a minimal shift with $h(\Z,\Sigma)=0$. Then $\aut(\Z,\Sigma)$ is sofic.
\end{theorem}
Soficity, as defined by Gromov \cite{gromov1999endomorphisms} (see also Weiss~\cite{weiss2000sofic}) is a joint weakening of amenability and residual finiteness, and so this result, in a weak sense, supports the aforementioned conjecture that these automorphism groups are amenable. 

\bigskip
\subsection*{Acknowledgments}
We would like to thank Lewis Bowen, Byrna Kra and Anthony Quas for helpful comments and suggestions.

\section{Proofs}
Let $G$ be a countable group, $A$ a finite alphabet and $(G,\Sigma)$ a subshift of $(G,A^G)$. Let $F$ be a finite subset of $G$. The restriction of $\sigma \in \Sigma$ to $F$ is denoted by $\sigma_F \colon F \to A$. We denote
\begin{align*}
    \Sigma_F = \{\sigma_F\,:\, \sigma \in \Sigma\},
\end{align*}
and denote the growth function of $\Sigma$ by
\begin{align*}
    N_\Sigma(F) = |\Sigma_F|.
\end{align*}


\begin{proposition}
\label{prop:growth}
 Let $G$ be a countable group, and let  $(F_n)_n$ be an increasing sequence of finite subsets of $G$ with $\cup_n F_n=G$. Let $(G,\Sigma)$ be a shift with the property that for every finite $K \subset G$ it holds that
 \begin{align*}
     \liminf_n \frac{N_\Sigma\left(\cup_{g \in K}g F_n\right)}{N_\Sigma(F_n)} = 1.
 \end{align*}
 Then $(G,\Sigma)$ admits a characteristic measure.
\end{proposition}
If $G$ is in addition amenable then $(G,\Sigma)$ admits a characteristic invariant measure. To see this, note that the set of characteristic measures is a compact, convex subset of the Borel measures on $\Sigma$. The group $G$ acts on this set, since for any characteristic $\nu$, $g \in G$ and $\varphi \in \aut(G,\Sigma)$ it holds that  $\varphi (g \nu) = g \varphi(\nu) = g \nu$. Since $G$ is amenable this action must have a fixed point, which is the desired characteristic invariant measure.

The proof of Proposition~\ref{prop:growth} will use the notion of a {\em memory set}. Given $\varphi \in \aut(G,\Sigma)$, there is some finite $K \subset G$ and a map $\Phi \colon A^K \to A$ such that
\begin{align*}
    [\varphi(\sigma)](g) = \Phi\left((g^{-1} \sigma)_{K}\right).
\end{align*}
The set $K$ is called a {\em memory set} of $\varphi$; see, e.g., \cite[p.\ 6]{ceccherini2010cellular}. We can assume  without loss of generality that $K$ contains the identity.

\begin{proof}[Proof of Proposition~\ref{prop:growth}]
For each $n$, let $\pi_n \colon \Sigma \to A^{F_n}$ be the restriction map $\sigma \mapsto \sigma_{F_n}$, so that $\pi_n(\Sigma) = \Sigma_{F_n}$. Let $S_n \subset \Sigma$ be a set of representatives of the set $\{\pi_n^{-1}(\sigma_{F_n})\,:\, \sigma \in \Sigma\}$ of preimages of $\pi_n$. Hence $\pi_n(S_n) = \Sigma_{F_n}$ and $|S_n| = |\Sigma_{F_n}| = N_\Sigma(F_n)$. 

Let $\nu_n$ be the uniform measure over $S_n$, and let $\nu$ be any weak limit of a subsequence of $(\nu_n)_n$; such a limit exists by compactness. We will show that $\nu$ is characteristic. 

Fix $\varphi \in \aut(G,\Sigma)$. Let $K \subset G$ be a memory set of $\varphi$, and assume it contains the identity. There is thus $\Phi \colon A^K \to A$ such that $[\varphi(\sigma)](g) = \Phi\left((g^{-1} \sigma)_{K}\right)$. Denote
\begin{align*}
    \tilde F_n = \bigcup_{g \in K} F_n g.
\end{align*}
Let $\tilde S_n = \{\sigma_{\tilde F_n}\,:\, \sigma \in S_n\}$ be the set of projections of the elements of $S_n$ to $\tilde F_n$. Since $\tilde F_n$ contains $F_n$ it follows that $|S_n| = |\tilde S_n|$.

Define $\varphi' \colon \Sigma_{\tilde F_n} \to \Sigma_{F_n}$ by
\begin{align*}
    [\varphi'(\sigma)](g) = \Phi\left((g^{-1} \sigma)_{K}\right),
\end{align*}
for $g \in F_n$.

By the definition of $\tilde F_n$ this is well defined, and moreover $\varphi(\sigma)_{F_n} = \varphi'(\sigma_{\tilde F_n})$; that is, $\varphi'$ maps the restriction of $\sigma$ to $\tilde F_n$ to the restriction of $\varphi(\sigma)$ to $F_n$. Hence $\varphi(S_n)_{F_n} = \varphi'(\tilde S_n)$. Also, $\varphi'$ is onto and so there is a subset $R_n \subseteq \Sigma_{\tilde F_n}$  such that the restriction of $\varphi'$ to $R_n$ is a bijection from $R_n$ to $\Sigma_{F_n}$.

For every $\varepsilon > 0$, we can, by the claim hypothesis, take $n$ to be large enough so that $N_\Sigma(F_n) \geq (1-\varepsilon)N_\Sigma(\tilde F_n)$.
Then $R_n$ and $\tilde S_n$ are both of size $N_\Sigma(F_n) \geq (1-\varepsilon)N_\Sigma(\tilde F_n)$. Since their union is contained in $\Sigma_{\tilde F_n}$ and is thus of size at most $N_\Sigma(\tilde F_n)$, their intersection is of size at least  $(1-2\varepsilon)N_\Sigma(\tilde F_n)$. Since
$$
\varphi(S_n)_{F_n} = \varphi'(\tilde S_n) \supseteq \varphi'(\tilde S_n \cap R_n),
$$ 
and since $\varphi'$ is a bijection when restricted to $R_n$,  $\varphi(S_n)$ is also of size at least $(1-2\varepsilon)N_\Sigma(\tilde F_n)$, which is at least $(1-2\varepsilon)N_\Sigma(F_n)$. 

Since $\nu_n$ is the uniform distribution on $S_n$, it follows that the push-forward measures $\pi_n(\nu_n)$ and $\pi_n(\varphi(\nu_n))$ differ by at most $2\varepsilon$ in total variation.  Since the sequence $(F_n)_n$ is increasing, this implies that for all $m \leq n$ it also holds that $\pi_m(\nu_n)$ and $\pi_m(\varphi(\nu_n))$  differ by at most $2\varepsilon$. Thus for each $m$, $\pi_m(\nu)$ and $\pi_m(\varphi(\nu))$ are identical, and so $\varphi(\nu) = \nu$, since $\cup_n F_n = G$, and so the cylinder sets defined by the restrictions $(\pi_m)_m$ form a clopen basis for the Borel $\sigma$-algebra. We have thus shown that $\nu$ is characteristic.

\end{proof}

Using Proposition~\ref{prop:growth}, the proof of our main result is straightforward.
\begin{proof}[Proof of Theorem~\ref{thm:growth}]
Denote $L(r) = \log N_\Sigma(B_r)$. By the claim hypothesis, there is a sequence $(r_k)_k$ such that $\lim_k L(r_k)/r_k = 0$. Thus, and because $L(r)$ is increasing, there is another subsequence $r_n$ such that for every $i > 0$ 
\begin{align*}
    \lim_\ell L(r_n+i) - L(r_n) = 0.
\end{align*}
Hence if we set $F_n = B_{r_n}$ then the conditions of Proposition~\ref{prop:growth} are satisfied, and thus the conclusion follows.
\end{proof}

Theorem~\ref{cor:abelian} is a corollary of the following more general statement.
\begin{theorem}
\label{thm:Sofic}
 Let $G$ be a countable group, and let  $(F_n)_n$ be an increasing sequence of finite subsets of $G$ with $\cup_n F_n=G$. Let $(G,\Sigma)$ be a minimal shift with the property that for every finite $K \subset G$ it holds that
 \begin{align*}
     \liminf_n \frac{N_\Sigma\left(\cup_{g \in K}g F_n\right)}{N_\Sigma(F_n)} = 1.
 \end{align*}
 Then $\aut(G,\Sigma)$ is sofic.
\end{theorem}

The following lemma will serve as our working definition of a sofic group; the reduction to the usual definition is straighforward (see, e.g., \cite[Lemma 2.1]{juschenko}). A {\em partially defined map} from a set $A$ to $A$ is a map from a subset of $A$ into $A$.
\begin{lemma}
\label{lem:Kate}
Let $H$ be a countable group. Suppose that for all finite subsets $\Phi\subset H$ and all $\eps>0$ we have a finite set $A$ and a map $g \mapsto \tilde g$ that assigns to each $g \in \Phi$ a partially defined map $\tilde g$ from $A$ to $A$  which satisfies the following four conditions:
\begin{enumerate}
    \item for every $g\in \Phi$ there is a subset $A_g \subset A$ with $|A \setminus A_g|/|A|<\eps$, such that the map $\tilde g$ is defined and injective on $A_g$.
    \item
    For the identity element $e \in G$, $\tilde e$ is the identity map wherever it is defined.
    \item
    $\widetilde{gh}(a)=\tilde g(\tilde h(a))$ whenever all three are defined.
    \item
    If there is some $a \in A$ such that $\tilde g(a)=(a)$, then $g$ is the identity.
\end{enumerate}
Then $H$ is sofic.
\end{lemma}
We will need the following compactness lemma. 
\begin{lemma}
\label{lem:compact}
Let $\varphi$ be  an automorphism of a subshift $(G,\Sigma)$ such that $\varphi(\sigma) \neq \sigma$ for all $\sigma \in \Sigma$. Then there is some finite set $K \subset G$ such that for all $\sigma \in \Sigma$ the restrictions $\sigma_{K}$ and $\varphi(\sigma)_{K}$ differ.
\end{lemma}
\begin{proof}
Let $(F_n)_n$ be an increasing sequence of finite subsets of $G$ with $\cup_n F_n=G$. Assume towards a contradiction that for each $n$ there is a $\sigma^n \in \Sigma$ such that $\sigma^n_{F_n} = \varphi(\sigma^n)_{F_n}$. Assume without loss of generality that the sequence $(\sigma^n)_n$ converges to $\sigma$. Since the sequence $(F_n)_n$ is increasing, $\varphi(\sigma)_{F_n} = \sigma_{F_n}$ for all $n$. Hence $\varphi(\sigma)=\sigma$, since $(F_n)_n$ exhausts $G$. This is in contradiction to our assumption that $\varphi$ has no fixed points.
\end{proof}
\begin{proof}[Proof of Theorem~\ref{thm:Sofic}]
Let $\Phi$ be a finite subset of $\aut(G,\Sigma)$ which includes the identity. Fix $1>\eps>0$. Let $K$ be a finite subset of $G$ that contains the memory sets\footnote{See the proof of Proposition~\ref{prop:growth} for the definition of a memory set.} of all $\varphi \in \Phi$.

Since $(S,\Sigma)$ is minimal, $\varphi(\sigma) \neq \sigma$ for every $\sigma \in \Sigma$ and non-trivial $\varphi$. To see this, note that if the set of fixed points of $\varphi$ is non-empty then it is a subshift, and so, by minimality, must be all of $\Sigma$. Accordingly, by Lemma~\ref{lem:compact}, we can enlarge $K$ (while keeping it finite) so that $\sigma_K \neq \varphi(\sigma)_K$ for all $\sigma \in \Sigma$ and $\varphi \in \Phi$.

To prove the claim we proceed to find partially defined maps which satisfy the assumptions in Lemma~\ref{lem:Kate} for $\Phi,\eps$. Choose $k$ large enough so that $N_\Sigma(\cup_{g \in K}g F_{k})/N_\Sigma(F_k)<1+\eps$. Denote $F = F_k$ and $\tilde F = \cup_{g \in K}g F_{k}$.

For every $\varphi \in \Phi$, there is a natural map $\varphi' \colon \Sigma_{\tilde F} \to \Sigma_{F}$ which, given $\sigma \in \Sigma$, maps the configuration $\sigma_{\tilde F}$ to the configuration $\varphi(\sigma)_{F}$. This is well defined, since $K$ contains the memory set of $\varphi$, and hence $\varphi(\sigma)_F$ is determined by $\sigma_{\tilde F}$.

Since $\varphi$ is an automorphism, $\varphi'$ is surjective. Now we set $A$ to be $\Sigma_{\tilde F}$ and let the partially defined map $\tilde\varphi$ from $A$ to $A$ be given by $\tilde\varphi(a) = b$ whenever there exists a $\sigma \in \Sigma$ such that $a = \sigma_{\tilde F}$, and $b$ is the unique element of $A = \Sigma_{\tilde F}$ whose projection on $\Sigma_{F}$ is $\varphi'(a)$. This map is undefined when uniqueness fails.

We now prove that this map has the four properties required by Lemma~\ref{lem:Kate}.
\begin{enumerate}
    \item Since the projection map $\pi:\Sigma_{\tilde F}\to \Sigma_{F}$ is surjective, and since $N_\Sigma(\tilde F)/N_\Sigma(F)<1+\eps$ there can be at most $\eps N_\Sigma(F)$ many elements in $\Sigma_F$ with more than one extension to $\Sigma_{\tilde F}$. Thus $\pi^{-1}$ is one-to-one on a $1-\eps$ fraction of $\Sigma_F$. Since $\varphi':\Sigma_{\tilde F}\to \Sigma_F$ is surjective, it follows that $\tilde\varphi$ is defined on a $(1-\eps)/(1+\eps)$ fraction of $A = \Sigma_{\tilde F}$.

    \item
    If $\sigma_{F} \in \Sigma_F$ has a unique extension to $\Sigma_{\tilde F}$ then that extension must be $\sigma_{\tilde F}$. Applying this to the identity of $\aut(G,\Sigma)$ yields the desired condition.
    \item
    Suppose $\tilde\psi(a)$, $\tilde\varphi(\tilde \psi(a))$ and $\widetilde{\varphi\psi}(a)$ are all defined. 
    We show that $\tilde\varphi(\tilde\psi(a))=\widetilde{\varphi\psi}(a)$. 
    
    Note that for any $\eta \in \Phi$ and $\sigma_{\tilde F} \in A$, if $\tilde\eta(\sigma_{\tilde F})$ is defined, then $\tilde\eta(\sigma_{\tilde F}) = \eta(\sigma)_{\tilde F}$. Applying this to $\psi$, $\varphi$ and $\varphi\psi$ we get that for $a = \sigma_{\tilde F}$
    \begin{align*}
        \tilde\varphi(\tilde\psi(\sigma_{\tilde F})) = \tilde\varphi(\psi(\sigma)_{\tilde F}) = \varphi\psi(\sigma)_{\tilde F} = \widetilde{\varphi\psi}(\sigma_{\tilde F}).
    \end{align*}

    
    \item 
    The fourth condition follows from the fact that $K \subseteq \tilde F$, and the defining property of $K$ that ensures that $\sigma_K$ and $\varphi(\sigma)_K$ differ.
\end{enumerate}
We have thus proved that all of the conditions of Lemma~\ref{lem:Kate} hold, and so and the group is sofic. 
\end{proof}
Theorem~\ref{cor:abelian} is an easy corollary of  Theorem~\ref{thm:Sofic}, as, by the same argument as in the proof of Theorem~\ref{thm:zero-ent}, every zero entropy subshift must satisfy
\begin{align*}
     \liminf_n \frac{N_\Sigma\left(\cup_{g \in K}g F_n\right)}{N_\Sigma(F_n)} = 1.
\end{align*}

\bibliography{refs}
\end{document}